\documentclass[11pt]{amsart}

\usepackage{amsfonts}
\usepackage{amsmath}

\newtheorem{theorem}{Theorem}[section]
\newtheorem{lemma}[theorem]{Lemma}
\newtheorem{proposition}[theorem]{Proposition}

\begin{document}

\title[Unimodality of the independence polynomials of non-regular caterpillars]{Unimodality of the independence polynomials of non-regular caterpillars}

\author[Patrick Bahls]{Patrick Bahls}

\author[Bailey Ethridge]{Bailey Ethridge}

\author[Levente Szabo]{Levente Szabo}

\address{Department of Mathematics \\ University of North Carolina \\ Asheville, NC 28804, USA}

\email{pbahls@unca.edu}

\email{bethridg@unca.edu}

\email{lszabo@unca.edu}

\keywords{independence polynomial, tree, unimodal}

\subjclass[2010]{05C31,05C69}

\begin{abstract}
The \textit{independence polynomial} $I(G,x)$ of a graph $G$ is the polynomial in variable $x$ in which the coefficient $a_n$ on $x^n$ gives the number of independent subsets $S \subseteq V(G)$ of vertices of $G$ such that $|S|=n$. $I(G,x)$ is \textit{unimodal} if there is an index $\mu$ such that $a_0 \leq a_1 \leq \cdots \leq a_{\mu-1} \leq a_{\mu} \geq a_{\mu+1} \geq \cdots \geq a_{d-1} \geq a_d$. While the independence polynomials of many families of graphs with highly regular structure are known to be unimodal, little is known about less regularly-structured graphs. We analyze the independence polynomials of a large infinite family of trees without regular structure and show that these polynomials are unimodal through a combinatorial analysis of the polynomials' coefficients.
\end{abstract}

\maketitle

\section{Introduction} \label{secIntroduction}

\noindent Let $G=(V,E)$ be a graph with vertices $V$ and edges $E$. An \textit{independent set} in $G$ is a subset of $V$ in which no two distinct vertices are adjacent. (In other words, the shortest path between each pair of vertices in an independent set is at least length two.)  The \textit{independence number} of $G$, denoted $\alpha(G)$, is the cardinality of a largest independent set in $V$.  The \textit{independence polynomial}, $I(G;x)$, of $G$ is defined by \[ I(G;x) = \sum_{n=0}^{\alpha(G)} a_nx^n, \] \noindent where $a_n$ is the number of independent subsets of $V$ of cardinality $n$.

We say that a sequence $\{a_0,a_1,...,a_d\}$ is \textit{unimodal} if there is some index $\mu$ such that $a_0 \leq a_1 \leq \cdots a_{\mu-1} \leq a_{\mu} \geq a_{\mu+1} \geq \cdots \geq a_{d-1} \geq a_d$; in this case we say that $\mu$ is a \textit{mode} of the sequence. A polynomial is said to be unimodal if its coefficient sequence is unimodal. Throughout this article, we will abuse the terminology and say that a graph is unimodal if its independence polynomial is unimodal. We will further say that a sequence $\{a_0,a_1,...,a_d\}$ (or, analogously, polynomial or graph) is \textit{strictly} unimodal if it has a unique mode $\mu$ and if \[ a_0 < a_1 < \cdots < a_{\mu-1} < a_{\mu} > a_{\mu+1} > \cdots > a_{d-1} > a_d. \]

The unimodality of various families of graphs has been the focus of a large amount of study. The survey \cite{LevitMandrescu2} provides an overview of a number of early results concerning unimodality and related properties (\textit{e.g.}, symmetry and logarithmic concavity). More recent studies include [2--7,12--15,17--19]. Despite the considerable attention paid to unimodality, little is known about the unimodality of any but the most regular families of graphs. In particular, little progress has been made in answering the following simple question, first posed by Alavi, Malde, Schwenk, and Erd\H{o}s in \cite{AlaviMaldeSchwenkErdos}: is the independence polynomial of an arbitrary tree unimodal?

A graph $T$ is called a \textit{tree} if it is connected and acyclic. Further, a tree $T$ is called a \textit{caterpillar} if the collection of all of $T$'s vertices of degree at least $2$ forms a path $P$, which we call the \textit{spine} of the caterpillar. Several authors have investigated the independence polynomials of ``regular'' caterpillars, in which the subtree of $T$ pendant at each vertex of $T$'s spine is identical. Beginning with \cite{ZhuZF}, Zhu proves that caterpillars in which every spine vertex has two pendant edges is unimodal. In fact, several studies of regular caterpillars encompass more general path-like graphs; see, for example, \cite{Bahls} and \cite{BahlsSalazar}. More recently, Galvin and Hilyard \cite{GalvinHilyard} investigate the behavior of $I(T,x)$ for a family of ``semi-regular'' caterpillar-like graphs, in which the subtrees pendant at the vertices in the graph's spine are not all identical to one another but rather alternate in a regular pattern.

Here we consider much more general caterpillars. Let $\vec{m} = (m_1,m_2,...)$ be a sequence of natural numbers and let $T(\vec{m},n)$ be the caterpillar with $n$ vertices on its spine, the $k$th of which has $m_k$ pendant edges.

\begin{theorem} \label{mainTheorem}
Let $T(\vec{m},n)$ be defined as above for $\vec{m} = (m_1,m_2,...)$ such that
\begin{enumerate}
\item $m_k \leq m_{k+1}$ for all $1 \leq k < n$,
\item $3 \leq m_1 < m_2$ and $m_3 < m_4$, and
\item $2(m_1+m_3+\cdots+m_k) < 3(m_2+m_4+\cdots+m_{k-1})$ for $k$ odd and $2(m_2+m_4+\cdots+m_k) < 3(m_1+m_3+\cdots+m_{k-1})$ for $k$ even.
\end{enumerate}
\noindent Then $I(T(\vec{m},n),x)$ is unimodal with mode $\mu_n \in \left\{ \lfloor \frac{d_n}{2} \rfloor, \lceil \frac{d_n}{2} \rceil \right\}$, where \[ d_n = \deg(I(T(\vec{m},n))) = \sum_{k=1}^n m_k. \]
\end{theorem}

Condition~(3) is a technical one that follows if $m_k$ does not grow too quickly. For instance, when $\vec{m}$ is non-decreasing, Condition~(3) holds whenever $m_k \leq m_1+m_3+\cdots+m_{k-1}$ when $k$ is even and $m_k \leq m_2+m_4+\cdots+m_{k-1}$ when $k$ is odd.

In the following section, we establish several important lemmas and in Section~\ref{secProofMain} we prove Theorem~\ref{mainTheorem}.

\section{The relationship between $q(x)$ and $(1+x)^t q(x)$} \label{secLemmas}

\noindent Our proofs depend on careful analysis of products of the form $(1+x)^t q(x)$ where $q$ is a polynomial whose coefficients are well understood. We will say that the strictly unimodal polynomial $q(x) = \sum_{i=0}^d b_ix^i$ with unique mode $\mu$ is \textit{left-dominant} (abbreviated LD) if \[ b_{\mu} > b_{\mu-1} > b_{\mu+1} > \cdots > b_1 > b_{2\mu-1} > b_0 \geq b_{2\mu} \] \noindent when $\mu \leq \frac{d}{2}$ and \[ b_{\mu} > b_{\mu-1} > b_{\mu+1} > \cdots > b_{2\mu-d+1} > b_{d-1} > b_{2\mu-d} \geq b_d \] \noindent when $\mu > \frac{d}{2}$. Similarly, we will say that $q$ is \textit{right-dominant} (abbreviated RD) if \[ b_{\mu} > b_{\mu+1} > b_{\mu-1} > \cdots > b_{2\mu-1} > b_1 > b_{2\mu} \geq b_0 \] \noindent when $\mu \leq \frac{d}{2}$ and \[ b_{\mu} > b_{\mu+1} > b_{\mu-1} > \cdots > b_{d-1} > b_{2\mu-d+1} > b_d \geq  b_{2\mu-d} \] \noindent when $\mu > \frac{d}{2}$. We will say that $q$ is \textit{weakly} LD (or \textit{weakly} RD) if the respective ordering permits equality to hold between terms strictly less than $b_{\mu}$, rather than requiring strict inequality. Finally, assuming that $q$ is both strictly unimodal and either LD or RD, we will say that $q$ is \textit{balanced} if $\mu = \lceil \frac{d}{2} \rceil$ when $q$ is LD and $\mu = \lfloor \frac{d}{2} \rfloor$ when $q$ is RD.

We now examine how strict unimodality and left- and right-dominance are affected when multiplying a polynomial $q$ by powers of $1+x$:

\begin{lemma} \label{alternatingSUmodesEvenBalanced}
Suppose that $q(x) = \sum_{i=0}^d b_ix^i$ is a balanced strictly unimodal polynomial with mode $\mu$. Let $t \in \mathbb{N}$ such that $t \leq \mu$, and suppose $(1+x)^t q(x) = \sum_{i=0}^{d+t} \beta_ix^i$.
\begin{enumerate}
\item If $t$ is even and $q$ is weakly LD (respectively, weakly RD), then $(1+x)^t q(x)$ is a balanced weakly LD (respectively, weakly RD) strictly unimodal polynomial with unique mode $\mu+\frac{t}{2}$.
\item If $t$ is odd and $q$ is weakly LD (respectively, weakly RD), then $(1+x)^t q(x)$ is a balanced weakly RD (respectively, weakly LD) strictly unimodal polynomial with unique mode $\mu+\frac{t-1}{2}$ (respectively, $\mu+\frac{t+1}{2}$).
\end{enumerate}

\noindent Moreover, if $q$ is LD or RD (and not merely weakly LD or weakly RD), then $(1+x)^t q(x)$ is LD or RD, accordingly as above.
\end{lemma} 

\begin{proof}
Let us assume that $t$ is even and $q$ is weakly LD; the remaining three cases are proven analogously.

Since \[ \beta_k = \sum_{i=0}^k {t \choose i} b_{k-i}\] \noindent when $0 \leq k \leq t$ and \[ \beta_{d+k} = \sum_{i=k}^t {t \choose i} b_{d+k-i} \] \noindent when $d \leq k \leq d+t$, the strict unimodality of $q$ shows that $\beta_0 \leq \beta_1 \leq \cdots \leq \beta_t$ and $\beta_d \geq \beta_{d+1} \geq \cdots \geq \beta_{d+t}$. When $t \leq k \leq d$, \[ \beta_k = \sum_{i=0}^t {t \choose i} b_{k-i}, \] \noindent and straightforward computation thus yields \[ \beta_{\mu+t/2}-\beta_{\mu+t/2-1} = \left[ \displaystyle\sum_{i=0}^{t/2-1} \left( {t \choose i+1} - {t \choose i} \right) \left( b_{\mu+t/2-i-1} - b_{\mu-t/2+i} \right) \right] + b_{\mu+t/2}-b_{\mu-t/2-1}. \] \noindent Because $q$ is LD, every parenthesized term in this expression is non-negative. Moreover, the term corresponding to $i=\frac{t}{2}-1$ involves the strictly positive difference $b_{\mu}-b_{\mu-1}$. Therefore the sum must be strictly positive, and $\beta_{\mu+t/2-1}<\beta_{\mu+t/2}$. Similarly, \[ \beta_{\mu+t/2}-\beta_{\mu+t/2+1} = \left[ \displaystyle\sum_{i=0}^{t/2-1} \left( {t \choose i+1} - {t \choose i} \right) \left( b_{\mu-t/2+i+1}-b_{\mu+t/2-i} \right) \right] + b_{\mu-t/2} - b_{\mu+t/2+1}, \] \noindent in which, again, every parenthesized term is non-negative and the term corresponding to $i=\frac{t}{2}-1$ is strictly positive, showing that $\beta_{\mu+t/2+1}<\beta_{\mu+t/2}$ and establishing that $\mu+\frac{t}{2}$ is the unique mode of $(1+x)^t q(x)$.

Completely analogous computations show that $\beta_k < \beta_{k+1}$ whenever $k \in [t,\mu+t/2-2]$ and $\beta_k > \beta_{k+1}$ whenever $k \in [\mu+t/2+1,d-1]$. Together with the above inequalities, these inequalities establish the strict unimodality of $(1+x)^t q(x)$. 

To establish weak left dominance, we must consider the differences $\beta_{\mu+t/2-s}-\beta_{\mu+t/2+s}$ for $s \leq \mu+\frac{t}{2}$. Let us first assume that $\frac{t}{2} \leq s \leq \mu-\frac{t}{2}$. Expanding in a similar manner as above, we obtain \[ \beta_{\mu+t/2-s}-\beta_{\mu+t/2+s} = \displaystyle\sum_{i=0}^t {t \choose i} \left( b_{\mu-(s+t/2-i)} - b_{\mu+(s+t/2-i)} \right). \] \noindent Since $q$ is weakly LD, every parenthesized term in this expression is non-negative, establishing $\beta_{\mu+t/2+s} \leq \beta_{\mu+t/2-s}$. If $s<t/2$, a different arrangement of the terms in the difference gives \[ \begin{array}{rcl} \beta_{\mu+t/2-s}-\beta_{\mu+t/2+s} & = & \displaystyle\sum_{i=s}^{t/2-1} \left( {t \choose i+s} - {t \choose i-s} \right) \left( b_{\mu-(t/2-i)}-b_{\mu+(t/2-i)} \right) \\ & + & \displaystyle\sum_{i=0}^{2s-1} {t \choose i} \left( b_{\mu-(t/2+s+i)}-b_{\mu+(t/2+s+i)} \right), \\ \end{array} \] \noindent where once more every parenthesized term is non-negative, establishing $\beta_{\mu+t/2+s} \leq \beta_{\mu+t/2-s}$ in this case as well. Finally, consider $s$ such that $\mu-\frac{t}{2}+1 \leq s \leq \mu+\frac{t}{2}$. Because $(1+x)^t q(x)$ is now known to be balanced, establishing weak left dominance for these values $s$ is equivalent to showing $\beta_{d+t-j} \leq \beta_j$ for $0 \leq j \leq t-1$. However, \[ \beta_j - \beta_{d+t-j} = \displaystyle\sum_{i=0}^j {t \choose i} (b_{j-i}-b_{d-(j-i)}), \] \noindent in which, yet again, every parenthesized term is non-negative because $q$ is itself balanced and weakly LD.

Very similar arguments show that $\beta_{\mu+t/2-(s+1)} \leq \beta_{\mu+t/2+s}$ for all $s$. Together these inequalities show that $(1+x)^tq(x)$ is weakly LD, as desired. Moreover, we note that were $q$ to be LD and not merely weakly LD, all of the above expressions would involve strictly positive terms and not merely non-negative ones, showing that $(1+x)^tq(x)$ would be LD as well. 

As mentioned above, the proofs in case $q$ is (weakly) RD, or in which $t$ is odd, are analogous.
\end{proof}

We note that the proof above can easily be modified to show that the lemma remains true when we replace $(1+x)^t$ with any even-degree polynomial $p(x) = \sum_{i=0}^t a_ix^i$ that is both unimodal and \textit{symmetric} (for which $a_j = a_{t-j}$ for any $j$, $0 \leq j \leq t$). Moreover, we may also show that if $q$ itself is symmetric and unimodal, then the polynomial $(1+x)^t q(x)$ is likewise symmetric and unimodal.

\bigskip

Similar techniques yield estimates for the differences $\beta_{k+1}-\beta_k$ and $\beta_k-\beta_{k+1}$ in terms of the differences $b_j-b_{j+1}$:

\begin{lemma} \label{boundOnDiffsEvenBalanced}
Suppose that $q(x) = \sum_{i=0}^d b_ix^i$ is a balanced strictly unimodal polynomial with mode $\mu$, and further that $q$ is either weakly LD or weakly RD. Let $t \in \mathbb{N}$ be such that $t \leq \mu$, and let $(1+x)^t q(x) = \sum_{i=0}^{d+t} \beta_ix^i$ and $\nu = \mu\bigl((1+x)^tq(x)\bigr) \in \{ \mu+\lfloor \frac{t}{2} \rfloor, \mu+\lceil \frac{t}{2} \rceil\}$.

\begin{enumerate}
\item Let $k \in [\mu+1,\nu-1]$. Then $\beta_{k+1}-\beta_k \geq \left( {t \choose k-j+1 } - {t \choose k-j} \right) (b_j-b_{j+1})$ for $j \in \left[ k-\lceil \frac{t}{2} \rceil+1,k \right]$ and $\beta_{k+1}-\beta_k \geq b_{k+1}-b_{k+2}$.

\item Let $k \in \left[ \nu,d-1 \right]$. Then $\beta_k-\beta_{k+1} \geq \left( {t \choose k-j} - {t \choose k-j-1} \right) (b_j-b_{j+1})$ for $j \in \left[ k-\lfloor \frac{t}{2} \rfloor,k-1 \right]$ and $\beta_k-\beta_{k+1} \geq b_k-b_{k+1}$.
\end{enumerate}
\end{lemma}

\begin{proof}
Let us consider Case (1) when $q$ is weakly LD. Then \[ \beta_{k+1}-\beta_k = \left[ \displaystyle\sum_{i=0}^{t/2-1} \left( {t \choose i+1} - {t \choose i} \right) (b_{k-i}-b_{k-t+i+1}) \right] + b_{k+1}-b_{k-t}. \] \noindent Isolating the $i$th term in the sum gives \[ \beta_{k+1}-\beta_k \geq \left( {t \choose i+1} - {t \choose i} \right) (b_{k-i}-b_{k-t+i+1}) \geq \left( {t \choose i+1} - {t \choose i} \right) (b_{k-i}-b_{k-i+1}), \] \noindent and letting $j=k-i$ gives the desired inequality. The inequality $\beta_{k+1}-\beta_k \geq b_{k+1}-b_{k+2}$ follows from isolating the lone term outside of the sum. Proving Case (1) for $q$ weakly RD is analogous.

Case (2) is proven in similar fashion, using the fact that \[ \beta_k-\beta_{k+1} = \left[ \displaystyle\sum_{i=0}^{t/2-1} \left( {t \choose i+1} - {t \choose i} \right) (b_{k-t+i+1}- b_{k-i}) \right] + b_{k-t}-b_{k+1}. \]
\end{proof}

We note that in general, the bounds given in Lemma~\ref{boundOnDiffsEvenBalanced} will offer very coarse estimates, given both the number of terms ignored in the proof above and the enormity of the coefficients on those terms. However, these bounds are sufficiently tight for our purposes, as we shall see in the next section.

\section{Applying the lemmas: unimodality of non-regular caterpillars} \label{secProofMain}

\noindent Recall that if $G$ is a graph and $S \subseteq V(G)$, then $G-S$ is defined to be the graph resulting from $G$ by removing all vertices in $S$ and all edges incident to at least one vertex in $S$. If $S = \{v\}$ comprises a single vertex, we may write $G-v$ for $G-\{v\}$. If $v \in V(G)$, then the \textit{closed neighborhood} of $v$, $N[v]$, is defined by $N[v] = \{u \in V(G) \ | \ uv \in E(V)\} \cup \{v\}$.

We first note a standard lemma that we will use frequently, often without explicit mention. Its proof is well-known and straightforward.

\begin{lemma} \label{recursiveIndPoly}
Suppose that $G$ is a graph and $v \in V(G)$. Then $I(G,x) = I(G-v,x)+xI(G-N[v],x)$.
\end{lemma}

Let us now recall the sequences $\vec{m} = (m_1,m_2,...)$ and corresponding caterpillars $T(\vec{m},n)$ defined in the introduction. We let $p_{\vec{m},n}(x) = I(T(\vec{m},n),x)$, and if $k(\vec{m},n)$ denotes the greatest power of $1+x$ evenly dividing $p_{\vec{m},n}$, we define $q_{\vec{m},n}$ by $p_{\vec{m},n} = (1+x)^{k(\vec{m},n)}q_{\vec{m},n}$. When $\vec{m}$ is understood, we may abbreviate $T(\vec{m},n)$ to $T_n$, \textit{etc.}

The following facts are proven by direct application of Lemma~\ref{recursiveIndPoly} (to the spine vertex $v$ with $m_n$ pendant edges) and straightforward inductions; compare the methods of [2--4], for instance:

\begin{proposition} \label{basicsOfMCs}
Let $\vec{m}$ be given a non-decreasing sequence of natural numbers. Then \[ p_n = \left\{ \begin{array}{ll} (1+x)^{m_1}+x & {\rm if} \ n = 1, \\ (1+x)^{m_1+m_2}+x\bigl((1+x)^{m_1}+(1+x)^{m_2}\bigr) & {\rm if} \ n=2, \ {\rm and} \\ (1+x)^{m_n}p_{n-1}(x) + x(1+x)^{m_{n-1}}p_{n-2}(x) & {\rm if} \ n \geq 3. \\ \end{array} \right. \] \noindent For all $n$, $\deg(p_n) = \sum_{i=1}^n m_i$. Moreover, $k_1=0$ and \[ k_n = \left\{ \begin{array}{ll} m_1+m_3+ \cdots + m_{n-1} & {\rm if} \ n \ {\rm is \ even \ and} \\ m_2+m_4+\cdots+m_{n-1} & {\rm if} \ n \ {\rm is \ odd,} \ n \geq 3,\end{array} \right. \] \noindent so that $k_{n+1} \geq k_n$ for all $n$, and \[ q_n = \left\{ \begin{array}{ll} (1+x)^{m_1}+x & {\rm if} \ n=1, \\ (1+x)^{m_2}+x(1+x)^{m_2-m_1}+x & {\rm if} \ n=2, \ {\rm and} \\ (1+x)^{k_{n+1}-k_n}q_{n-1}(x)+xq_{n-2}(x) & {\rm if} \ n \geq 3. \\ \end{array} \right. \]
\noindent Thus, for all $n$, \[ \deg(q_n) = \left\{ \begin{array}{ll} m_1+m_3+\cdots+m_n & {\rm if} \ n \ {\rm is \ odd \ and} \\ m_2+m_4+\cdots+m_n & {\rm if} \ n \ {\rm is \ even.} \\ \end{array} \right. \]
\end{proposition}

We now apply Lemmas~\ref{alternatingSUmodesEvenBalanced} and \ref{boundOnDiffsEvenBalanced} to show that under the right hypotheses $q_n(x)$ is balanced, strictly unimodal, and either LD or RD. Once this is done, one more application of Lemma~\ref{alternatingSUmodesEvenBalanced} will establish the same properties for $p_n(x) = I(T(\vec{m},n),x)$, thereby proving our main result, Theorem~\ref{mainTheorem}.

\begin{proposition} \label{mainTheoremRedux}
Let $\vec{m} = (m_1,m_2,...)$ be a sequence of natural numbers such that
\begin{enumerate}
\item $m_k \leq m_{k+1}$ for all $1 \leq k < n$,
\item $3 \leq m_1 < m_2$ and $m_3 < m_4$, and
\item $2(m_1+m_3+\cdots+m_k) < 3(m_2+m_4+\cdots+m_{k-1})$ for $k$ odd and $2(m_2+m_4+\cdots+m_k) < 3(m_1+m_3+\cdots+m_{k-1})$ for $k$ even.
\end{enumerate}
Then the polynomial $q_n(x)$ defined as above is balanced, strictly unimodal, and either LD or RD.
\end{proposition}

\begin{proof}
Straightforward computation shows that both $q_1$ and $q_2$ are balanced, with $q_1$ weakly LD and $q_2$ either LD or RD. The polynomial $q_1$ is always strictly unimodal, and $q_2$ is unimodal in general. It will be strictly unimodal except when $m_1$ is even, $m_1 \geq 6$, and $m_2=m_1+1$; in this case, $q_2$ has consecutive modes at $\frac{m_2-1}{2}$ and $\frac{m_2+1}{2}$. In any case, we may take these polynomials to be the base cases for an induction. Assume that we have shown our result for all $k \leq n-1$ for some $n \geq 3$.

Let $\mu = \mu(q_{n-2})$ and $\mu' = \mu(q_{n-1})$, and suppose \[ q_{n-2} = \sum a_ix^i, q_{n-1} = \sum b_ix^i, xq_{n-2} = \sum \alpha_ix^i, \ {\rm and} \ (1+x)^t q_{n-1} = \sum \beta_ix^i, \]  \noindent where \[ t = k_{n+1}-k_n = \left\{ \begin{array}{ll} (m_2-m_1) + \cdots + (m_n-m_{n-1}) & {\rm if} \ n \ {\rm is \ even \ and} \\ m_1 + (m_3-m_2) + \cdots + (m_n-m_{n-1}) & {\rm if} \ n \ {\rm is \ odd.} \\ \end{array} \right. \] \noindent Clearly, $\alpha_i = a_{i-1}$ for all $i \geq 1$, $\mu(xq_{n-2}) = \mu+1$, and Lemma~\ref{alternatingSUmodesEvenBalanced} implies that $(1+x)^tq_{n-1}$ is balaced and unimodal and either LD or RD. Moreover, if we let $\mu'' = \mu((1+x)^tq_{n-1})$, then $\mu'' = \mu'+\lfloor \frac{t}{2} \rfloor$ if $q_{n-1}$ is LD and $\mu'' = \mu'+\lceil \frac{t}{2} \rceil$ if $q_{n-1}$ is RD. (Note that Lemma~\ref{alternatingSUmodesEvenBalanced} applies because our third hypothesis on the values $m_i$ ensures that $t \leq \mu'$.) Finally, because $q_n = (1+x)^tq_{n-1}+xq_{n-2}$, $q_n = \sum (\alpha_i+\beta_i)x^i = \sum (a_{i-1}+\beta_i)x^i$.

Clearly $\alpha_k+\beta_k < \alpha_{k+1}+\beta_{k+1}$ for $k \leq \mu$ and $\alpha_k+\beta_k > \alpha_{k+1}+\beta_{k+1}$ for $k \geq \mu''$. Thus, to prove strict unimodality of $q_n$ we only need to consider $k \in \left[ \mu+1,\mu''-1 \right]$, and for such $k$ it suffices to show that $a_{k-1}-a_k < \beta_{k+1}-\beta_k$. First note that, by the definition of $b_j$ and $\beta_k$, Part (1) of Lemma~\ref{boundOnDiffsEvenBalanced} implies $\beta_{k+1}-\beta_k \geq (t-1)(b_k-b_{k+1})$. Moreover, because $q_{n-1} = (1+x)^{k_n-k_{n-1}}q_{n-2} + xq_{n-3}$, Part (2) of the same lemma implies $b_k-b_{k+1} \geq (t'-1)(a_{k-1}-a_k)$, where $t'=k_n-k_{n-1}$. (In fact, this last estimate ignores the non-negative contribution from the term $xq_{n-3}$, inclusion of which would only serve to increase the left-hand side of the last inequality.) Putting these together, we obtain $\beta_{k+1}-\beta_k \geq (t-1)(t'-1)(a_{k-1}-a_k)$ Moreover, our first two hypotheses on the values $m_i$ imply that $t,t' \geq 2$ whenever $n \geq 3$. Thus $(t-1)(t'-1) \geq 1$, and the inequality above implies that $\beta_{k+1}-\beta_k \geq a_{k-1}-a_k$, as desired.

Thus $q_n$ is strictly unimodal and \[ \mu(q_n) = \mu((1+x)^tq_{n-1}) = \left\{ \begin{array}{ll} \mu(q_{n-1})+\lfloor \frac{t}{2} \rfloor & {\rm if} \ q_n \ {\rm is \ RD \ and} \\ \mu(q_{n-1})+\lceil \frac{t}{2} \rceil & {\rm if} \ q_n \ {\rm is \ LD}, \\ \end{array} \right. \] \noindent so that $q_n$ is balanced as well. Moreover, Lemma~\ref{boundOnDiffsEvenBalanced} also implies that the addition of the terms $\alpha_k$ does not affect the LD or RD nature of $(1+x)^tq_{n-1}$. That is, $q_n = (1+x)^tq_{n-1}+xq_{n-2}$ remains LD or RD, accordingly.
\end{proof}

As noted before the statement of the proposition, the fact that $p_n(x) = I(T(\vec{m},n),x)$ is balanced, strictly unimodal (even in the case of $p_2$ when $m_1 \geq 6$ is even and $m_2=m_1+1$), and either LD or RD follows from one more application of Lemma~\ref{alternatingSUmodesEvenBalanced}, because $p_n(x) = (1+x)^{k_n}q_n(x)$ for all $n$. This completes our proof of Theorem~\ref{mainTheorem}.

We close by offering an explanation for our requirement that $\vec{m}$ be non-decreasing. Observe that for $n \geq 3$, $k_n = \min\{k_{n-1}+m_n,k_{n-2}+m_{n-1}\}$, so that if $\vec{m}$ were not non-decreasing, it could be that, for some $n$, $k_{n-1}+m_n<k_{n-2}+m_{n-1}$, giving us \[ q_n = q_{n-1}+x(1+x)^{k_{n-2}+m_{n-1}-k_n}q_{n-2}, \] \noindent in which case our fundamental result, Lemma~\ref{alternatingSUmodesEvenBalanced}, would not apply. Therefore without significant further analysis of the recursive construction of an arbitrary caterpillar's independence polynomial, we cannot accommodate sequences $\vec{m}$ in which $m_k > m_{k+1}$ for some $k$.

\end{document}